\documentclass[12pt,reqno]{amsart}
\usepackage{amsmath,amssymb,enumerate}
\usepackage[dvipdfmx]{graphicx}
\usepackage{color}
\usepackage{comment}
\usepackage[all]{xy}
\usepackage{bm}

\newtheorem{tm}{Theorem}[section]

\newtheorem{re}[tm]{Remark}
\newtheorem{df}[tm]{Definition}
\newtheorem{pr}[tm]{Proposition}

\makeatletter
\newcommand{\subscripts}[3]{%
  \@mathmeasure\z@\displaystyle{#2}%
  \global\setbox\@ne\vbox to\ht\z@{}\dp\@ne\dp\z@
  \setbox\tw@\box\@ne
  \@mathmeasure4\displaystyle{\copy\tw@_{#1}}%
  \@mathmeasure6\displaystyle{{#2}_{#3}}%
  \dimen@-\wd6 \advance\dimen@\wd4 \advance\dimen@\wd\z@
  \hbox to\dimen@{}\mathop{\kern-\dimen@\box4\box6}%
}
\makeatother
\newcommand{\R}{\mathbb{R}}
\newcommand{\N}{\mathbb{N}}

\newcommand{\dd}{\mathrm{d}}

\newcommand{\ve}{\varepsilon}
\newcommand{\dis}{\displaystyle}

\newcommand{\Var}{\mathrm{Var}}

\newcommand{\nn}{\nonumber}

\newcommand{\E}{\mathbb{E}}
\newcommand{\X}{\mathbf{X}}

\allowdisplaybreaks[3]

\makeatletter
 
 \@addtoreset{equation}{section}
\makeatother

\begin{document}
%%%%%%%%%%%%%%%%%%%%%%%%%%%%%%%%%%%%%%%%%%%
%\setlength{\baselineskip}{15.5pt}
%%%%%%%%%%%%%%%%%%%%%%%%%%%%%%%%%%%%%%%%%%%
\title[Limit theorems for iterates of the Sz\'asz--Mirakyan operator]
{
{ Limit theorems for iterates of the Sz\'asz--Mirakyan operator in probabilistic view}
}
%%%%%%%%%%%%%%%%%%%%%%%%%%%%
\author[J. Akahori]{Jir\^o Akahori}
\author[R. Namba]{Ryuya Namba}
\author[S. Semba]{Shunsuke Semba}

\address[J. Akahori]{Department of Mathematical Sciences,
College of Science and Engineering,
Ritsumeikan University, 1-1-1, Noji-higashi, Kusatsu, 525-8577, Japan}
\email{{\tt{akahori@se.ritsumei.ac.jp}}}

\address[R. Namba]{Department of Mathematics,
Faculty of Education,
Shizuoka University, 836, Ohya, Suruga-ku, 
Shizuoka, 525-8577, Japan}
\email{{\tt{namba.ryuya@shizuoka.ac.jp}}}

\address[S. Semba]{Graduate School of Science and Engineering,
Ritsumeikan University, 1-1-1, 
Noji-higashi, Kusatsu, 525-8577, Japan}
\email{{\tt{12s0523sun@gmail.com}}}

%%%%%%%%%%%%%%%%%%%%%%%%%%%%

\maketitle 
%
%
%%%%%%%%%%%%%%%%%%%%%%%%%%%%%%%%%%%%%%%%%%%%%%%%%%%%%%
\begin{abstract}
The Sz\'asz--Mirakyan operator is known as 
a positive linear operator which uniformly approximates 
a certain class of continuous functions on the half line. 
The purpose of the present paper is to find out 
limiting behaviors of the iterates of the Sz\'asz--Mirakyan operator 
in a probabilistic point of view. 
We show that the iterates of the Sz\'asz--Mirakyan operator
uniformly converges to a continuous semigroup generated by
a second order degenerate differential operator. 
A probabilistic interpretation of the convergence 
in terms of a discrete Markov chain constructed from the iterates 
and a limiting diffusion process on the half line is 
captured as well. 

\end{abstract}
%%%%%%%%%%%%%%%%%%%%%%%%%%%%%%%%%%%%%%%%%%%%%%%%%%%%%%

%\tableofcontents

%%%%%%%%%%%%%%%%%%%%%%%%%%%%%%%%%%%%%%%%%%%%%%%%%%%%%%%%%%%%%%%%%%%%%%%%%
\section{{\bf Introduction and main results}}
%%%%%%%%%%%%%%%%%%%%%%%%%%%%%%%%%%%%%%%%%%%%%%%%%%%%%%%%%%%%%%%%%%%%%%%%%

For $ n\in \N $, 
we define a linear operator 
$ B_n $ acting on $ C([0, 1]) $, 
the Banach space of all continuous functions on $ [0, 1] $, by 
    \begin{equation}\label{Bernstein op}
        B_n f(x) := \sum_{i=0}^n 
        \binom{n}{k} x^k (1-x)^{n-k} f\left(\frac{k}{n}\right),
        \qquad f \in C([0, 1]), \, x \in [0, 1].
    \end{equation}
The operator $ B_n $ was originally introduced by Bernstein
in \cite{Bernstein}. 
Therefore, it is called the {\it Bernstein operator}
and $ B_n f $ the {\it Bernstein polynomial} after his name. 
The Bernstein operator is famous for its 
remarkable property that  
$ B_n f $ approximates  
continuous functions uniformly in $ x \in [0, 1] $.

\begin{pr}[cf.~\cite{Bernstein}]
\label{Prop:Bernstein's thm}
    We have
    $$
        \lim_{n \to \infty}\max_{x \in [0, 1]}
        |B_n f(x) - f(x)|=0, 
        \qquad f \in C([0, 1]). 
    $$
\end{pr}

\noindent
It turns out that Proposition \ref{Prop:Bernstein's thm} 
provides a constructive proof 
of the celebrated {\it Weierstrass approximation theorem} in a probabilistic way. 
Indeed, the proof is an immediate consequence of the weak law of large numbers
and is referred to as an exercise of interest 
in many textbooks of probability theory (cf.~\cite[Example 5.15]{Klenke}). 
We also refer to e.g., \cite{Bustamante} for basic facts related to the 
Bernstein operator.

In focusing on iterates $ B_n^k $ of  $ B_n $ itself $k$ times, 
it is interesting to reveal the limiting behaviors of $ B_n^k f $
as $ n \to \infty $ and/or $ k \to \infty $. 
Kelisky and Rivlin  first studied such kind of limit theorems for 
$ B_n^k $ in \cite{KR} and obtained 
    $$
    \lim_{k \to \infty}\max_{x \in [0, 1]}
    \big|B_n^k f(x) - \{f(0)+(f(1)-f(0))x\}\big|=0, 
    \qquad n \in \mathbb{N}, \,\, f \in C([0, 1]),
    $$
which means that $ B_n^k f $ converges uniformly to 
a linear function which interpolates between $ f(0) $ and $ f(1) $
as $ k \to \infty $ with fixed $ n $. 
We also refer to \cite{Jachymski} for Kelisky--Rivlin type theorems 
for various kinds of positive linear operators. 

Next, let us consider the case where 
both $ k $ and $ n $ tend to infinity 
and the ratio $ k/n $ tends to some constant $t >0$. 
We fix $ n \in \mathbb{N} $.
For $ x \in [0, 1] $, let $ S_n=S_n(x) $
denote a random variable given by 
    \[
    \mathbb{P}(S_n(x)=k)=
    \binom{n}{k} x^k (1-x)^{n-k}, \qquad k=0, 1, 2, \dots, n,
    \]
and $ G_n=G_n(x) $ be a random function 
on $ [0, 1] $ defined by 
    $$
    G_n(x):=\frac{1}{n}S_n(x), \qquad x \in [0, 1].
    $$
Moreover, let $ \{G_n^k\}_{k=1}^\infty $  be a sequence of 
independent copies of $ G_n $.
For $ k \in \N, $ we define a random function $ H_n^k : [0, 1] \to \R $ by 
    $H_n^k:=G_n^k \circ G_n^{k-1} \circ \cdots \circ G_n^1.$
Then, it follows from the definition 
of $B_n$ that 
    \begin{equation}
    \label{Bernstein-Markov chain}
    \mathbb{E}[f(H_n^k(x))]
    =B_n^kf(x), \qquad x \in [0, 1], \, k=1, 2, \dots. 
    \end{equation}
Then, for every $ x \in [0, 1], $
we easily observe that the sequence 
$ \{H_n^k(x)\}_{k=1}^\infty $ is a time-homogeneous Markov chain 
with values in $ \{i/n \, | \, i=0, 1, 2, \dots, n\} $
whose one-step transition probability is given by 
    $$
    \mathbb{P}\left(
    H_n^{k+1}(x)=\frac{j}{n} \, \Bigg| \, 
    H_n^k(x)=\frac{i}{n}\right)
    =\binom{n}{j}\left(\frac{i}{n}\right)^j\left(1-\frac{i}{n}\right)^{n-j}, 
    \qquad i, j=0, 1, 2, \dots, n.
    $$
This describes the simplest stochastic model in mathematical biology, 
called the {\it Wright--Fisher model} for population genetics. 
We refer to e.g., \cite[Chapter 10]{EK} 
for more details in a probabilistic point of view. 
Besides, Konstantopoulos, Yuan and Zazanis  showed in \cite{KYZ} that 
the random curve defined by 
$ t \longmapsto H_n^{[nt]}(x), \, t \ge 0, $
weakly converges to a certain diffusion process with values in $ [0, 1] $. 
The precise statement is the following. 

\begin{pr}[cf.~{\cite[Theorem 3]{KYZ}}]
\label{Prop:KYZ-weak-conv}
   For $ x \in [0, 1], $
   let $ (\X_t(x))_{t \ge 0} $ be the diffusion process
   which solves the stochastic differential equation
   $$
   \dd \X_t(x)=\sqrt{\X_t(x)(1-\X_t(x))} \, \dd W_t, 
   \qquad \X_0(x)=x,
   $$
   where $ (W_t)_{t \ge 0} $ is a one-dimensional 
   standard Brownian motion. 
   Then, for every $ f \in C([0, 1]) $ and $ t \ge 0, $
   we have 
   $$
   \lim_{n \to \infty}\max_{x \in [0, 1]}
   \big|B_n^{[nt]}f(x) - \mathbb{E}[f(\X_t(x)]\big|=0.
   $$
\end{pr}
\noindent
Since the Markov chain $ \{H_n^{k}(x)\}_{k=1}^\infty $ is 
absorbing and it reaches the states 0 or 1 within finite times
with probability one, the limiting process $ (\X_t(x))_{t \ge 0} $
is also absorbed at the boundary of $ [0, 1] $ (cf.~\cite[Lemma 1]{KYZ}). 
The stochastic process $ (\X_t(x))_{t \ge 0} $
is called the {\it Wright--Fisher diffusion}, which is the 
most fundamental model in population genetics  
used so as to approximate the discrete Markov chain 
$ \{H_n^k(x)\}_{k=1}^\infty $. 
See e.g., \cite{Feller} for an early work on this topic. 

We note that Proposition \ref{Prop:KYZ-weak-conv} can be also 
proved by employing functional analytic techniques 
such as the celebrated {\it Trotter's approximation theorem} 
(cf.~\cite{Trotter}, see also \cite{Kurtz}). 
On the other hand, 
in \cite{KYZ}, the authors apply usual techniques 
in stochastic calculus 
in order to establish limit theorems
for the iterates of the Bernstein operator including 
Proposition \ref{Prop:KYZ-weak-conv}. 
Hence, they can reveal remarkable relations between 
the limiting behaviors of iterates of the Bernstein operator
and the stochastic phenomena behind the operator approximating 
continuous functions, 
which motivates a further study of interest.

So far, a large amount of generalizations of the Bernstein operator 
have been investigated extensively in various settings. 
There seems to be several directions 
to generalize the classical Bernstein operator \eqref{Bernstein op} 
in view of approximations of continuous functions. 
Among them, it should be most natural to consider generalizations
of \eqref{Bernstein op} to the infinite interval cases such as 
$[0, \infty)$ and $\R=(-\infty, \infty)$. 
We refer to e.g., \cite{AC},
for a generalization of \eqref{Bernstein op} to the case of $[0, \infty)$.

Let $C([0, \infty))$ be the linear space consisting of 
continuous functions defined on $ [0, \infty) $.
The present paper focuses on the so-called 
{\it Sz\'asz--Mirakyan operator} on $ [0, \infty) $,
which was originally introduced by
Mirakyan \cite{Mirakyan} and Sz\'asz \cite{Szasz} 
independently as a generalization of  \eqref{Bernstein op}
to the case of $ [0, \infty) $.

\begin{df}[Sz\'asz--Mirakyan operator]
For $ n \in \N $, we define the Sz\'asz--Mirakyan operator 
$ \mathcal{P}_n $ acting on $ C([0, \infty)) $ by
    \begin{equation}\label{Szasz-Mirakyan}
        \mathcal{P}_nf(x)
        :=\sum_{k=0}^\infty e^{-nx}\frac{(nx)^k}{k!}
        f\left(\frac{k}{n}\right), 
        \qquad f \in C([0, \infty)), \, x \in [0, \infty).
    \end{equation}
\end{df}

\noindent
Compared with the case of uniform approximations 
of continuous functions on $ [0, 1] $,
it is harder to treat the infinite interval cases 
since continuous functions on $ [0, \infty) $
such as polynomial functions are not always bounded. 
Therefore, we need to restrict ourselves to some function spaces weighted
by bounded functions so as to make the approximation 
of continuous functions by $ \mathcal{P}_n $
go well (see Proposition \ref{Prop:conv-SM-weighted}).

Before stating our main results, we need to fix some notations 
on several function spaces on $[0, \infty)$. 
We denote by $ C_b([0, \infty)) $ 
    the linear space of all bounded continuous functions on $ [0, \infty) $, 
    which is a Banach space
    with respect to the usual norm 
    $ \|f\|_\infty=\sup_{x \in [0, \infty)}|f(x)| $. 
    Let $ C_\infty([0, \infty)) $ be
    the linear space of all continuous functions 
    vanishing at infinity, 
    which is also a Banach subspace of $ C_b([0, \infty)) $. 
    For a weight function $ w \in C([0, \infty)) $
    satisfying $ w(x)>0, \, x \in [0, \infty) $, 
    we put 
      $C_\infty^w([0, \infty))
       :=\{f \in C([0, \infty)) \, | \, 
       wf \in C_\infty([0, \infty))\}.$
    It also becomes a Banach space when we endow it with the weighted norm 
    \[
    \|f\|_{\infty, w}:=\sup_{x \in [0, \infty)}
    |w(x)f(x)|, \qquad f \in C_\infty^w([0, \infty)). 
    \]
    For $1 \le r \le \infty$, 
    we denote by $ C^r([0, \infty)) $ 
    the linear space of all functions on $ [0, \infty) $
    having a continuous $r$-th derivative. 
    We also put
    $C_\infty^r([0, \infty))
    :=C_\infty([0, \infty)) \cap C^r([0, \infty)).$
    By $ \mathrm{Lip}([0, \infty)) $, we mean 
    the set of all Lipschitz continuous functions on  
    $ [0, \infty) $. We also put 
    \[
    \mathrm{Lip}(f):=\sup_{x, y \in [0, \infty), \, x \neq y}
    \frac{|f(x)-f(y)|}{|x-y|}, \qquad f \in \mathrm{Lip}([0, \infty)). 
    \]

For $ \alpha \ge 1 $, we define 
    \begin{equation}\label{weight}
    w_\alpha(x):=\frac{1}{1+x^\alpha}, \qquad 
    x \in [0, \infty).
    \end{equation}
The first main result of the present paper is the 
{\it Voronovskya-type theorem} for the sequence 
$ \{n(\mathcal{P}_n-I)\}_{n=1}^\infty $
on the function space weighted by $ w_\alpha $, 
where $I$ stands for the identity operator. 

\begin{tm}
\label{Thm:Voronovskaya-Szasz-Mirakyan}
Let $ \alpha \ge 1 $.
For every $ f \in C_\infty^{2}([0, \infty)) $, 
we have 
    \[
    \lim_{n \to \infty}
    \|n(\mathcal{P}_n f-f) - \mathcal{A}f\|_{\infty, w_{\alpha}}=0,
    \]
where $ \mathcal{A} $ is the degenerate differential operator
defined by
    \begin{equation}
        \label{generator-SM}
    \mathcal{A}f(x):=\begin{cases}
    \dfrac{x}{2}f''(x) & \text{if }x>0 \\
    0 & \text{if }x=0
    \end{cases}
    \end{equation}
for $ f \in D(\mathcal{A}) $, where 
the domain $ D(\mathcal{A}) $ of $ \mathcal{A} $
is given by 
    \begin{equation}
        \label{generator-SM-domain}
    D(\mathcal{A})=\Big\{
    f \in C_\infty^{w_\alpha}([0, \infty)) \cap C^2([0, \infty))
    \, \Big| \, \lim_{x \to 0+}xf''(x)
    =\lim_{x \to \infty}w_\alpha(x)(xf''(x))=0\Big\}.
    \end{equation}
Moreover, for $ f \in C^2_\infty([0, \infty)) $ with 
$ f'' \in \mathrm{Lip}([0, \infty)) $ and $ \alpha>3/2 $, 
we have 
    \[
    \|n(\mathcal{P}_n f-f) - \mathcal{A}f\|_{\infty, w_{\alpha}} \le 
    \frac{1}{6\sqrt{n}}M_\alpha\mathrm{Lip}(f''), \qquad n \in \N,
    \]
where 
    \[
    M_\alpha:=\frac{3^{3/4}(2\alpha-3)}{2\alpha}\left(\frac{3}{2\alpha-3}\right)^{3/2\alpha}
    +\frac{4\alpha-3}{4\alpha}\left(\frac{3}{4\alpha-3}\right)^{3/4\alpha}.
    \]
\end{tm}

\noindent
This theorem tells us a rate of convergence of 
$ \{n(\mathcal{P}_n-I)\}_{n=1}^\infty $
to the differential operator $\mathcal{A}$ with respect to the weighted norm. 
Due to the degeneracy of $\mathcal{A}$ at the boundary of $[0, \infty)$,
we need to take care of the boundary condition for $\mathcal{A}$, 
which is occasionally called the {\it Wentzell-type condition}. 

 Based on Theorem \ref{Thm:Voronovskaya-Szasz-Mirakyan}, 
we also establish the uniform convergence of the iterates of the 
Sz\'asz--Mirakyan operator to the $C_0$-semigroup generated by $\mathcal{A}$
with respect to the weighted norm,
together with its rate of convergence. 
The following is the second main result of the present paper. 

\begin{tm}
\label{Thm:semigroup-conv-SM}
For $ \alpha > 1 $, $ f \in C_\infty^{w_\alpha}([0, \infty)) $ 
and $ t \ge 0 $, we have 
    \[
    \lim_{n \to \infty}
    \|\mathcal{P}_n^{[nt]} f 
    - \mathsf{P}_t f\|_{\infty, w_\alpha}=0,
    \]
where $ ({\sf P}_t)_{t \ge 0} $ is a contraction $C_0$-semigroup
on $ C_\infty^{w_\alpha}([0, \infty)) $ generated by $ (\mathcal{A}, D(\mathcal{A})) $ defined by 
\eqref{generator-SM} and \eqref{generator-SM-domain}. 
Furthermore, let $ \mathcal{D}_0 $ be 
a subspace of $ C_\infty^2([0, \infty)) $ given by
    \[
    \mathcal{D}_0:=\Big\{f \in C_\infty^2([0, \infty)) \, \Big| \, 
    \text{$ f'' \in \mathrm{Lip}([0, \infty)) $, 
    $ {\sf P}_t f  \in C_\infty^2([0, \infty)) $ and 
    $ ({\sf P}_t f)''  \in \mathrm{Lip}([0, \infty)) $}\Big\}.
    \]
If $ f \in \mathcal{D}_0 $ and 
$ \alpha>3/2 $, we have
    \begin{equation}\label{SM-rate of conv}
    \begin{aligned}
    \|\mathcal{P}_n^{[nt]} f 
    - \mathsf{P}_t f\|_{\infty, w_\alpha}
    &\le \left(\sqrt{\frac{t}{n}}+\frac{1}{n}\right)
    \left(\|\mathcal{A}f\|_{\infty, w_\alpha}+\frac{1}{6\sqrt{n}}M_\alpha\mathrm{Lip}(f'')\right)\\
    &\hspace{1cm}+\int_0^t \frac{1}{6\sqrt{n}}
    M_\alpha\mathrm{Lip}\big(({\sf P}_sf)''\big) \, \dd s. 
    \end{aligned}
    \end{equation}

\end{tm}

We should emphasize that
our arguments in the proofs basically rely on 
the probabilistic approaches. 
In particular, the proof of Theorem \ref{Thm:semigroup-conv-SM} is given 
by the combination of the probabilistic approaches with
some results in functional analysis such as 
Trotter's approximation theorem (cf.~\cite{Trotter,Kurtz})
and a result on its rate of convergence (cf.~\cite{CT,CTa}).
On the other hand, 
one wonders if or not some probabilistic interpretation of 
the $C_0$-semigroup $({\sf P}_t)_{t \ge 0}$ can be given, 
since its infinitesimal generator $\mathcal{A}$ is a second order 
differential operator
and it may highly relate to a continuous stochastic process 
called a diffusion process. 
After the proofs of main results, we consider such a
probabilistic interpretation and obtain a relation between
$({\sf P}_t)_{t \ge 0}$ and a diffusion semigroup on $ C_\infty^{w_\alpha}([0, \infty)) $. 
Furthermore, we are going to describe this relation in terms of a 
time-homogeneous Markov chain constructed by the iterates of the
Sz\'asz--Mirakyan operator and a diffusion process 
captured as its scaling limit. 
See Theorem \ref{Thm:SM-diffusion}.

The rest of the present paper is organized as follows:
In Section \ref{Section:Szasz-Mirakyan}, we review 
several known facts about the approximation properties of 
the Sz\'asz--Mirakyan operator. 
We show in Section \ref{Section:proofs} our two main results, 
Theorems \ref{Thm:Voronovskaya-Szasz-Mirakyan} and \ref{Thm:semigroup-conv-SM}. 
Section \ref{Section:probab-approach} discusses an interesting 
interpretation of main theorems from probabilistic perspectives. 
In particular, we show in Theorem \ref{Thm:SM-diffusion}
that the $C_0$-semigroup obtained as a limit of the iterates of \eqref{Szasz-Mirakyan}
coincides with a certain diffusion semigroup. 
In Section \ref{Section:Conclusion}, we give not only a conclusion of the present paper
but also further possible directions of this study.

%%%%%%%%%%%%%%%%%%%%%%%%%%%%%%%%%%%%%%%%%%%%%%%%%%%%%%%%%%%%%%%%%%%%%%%%%
\section{{\bf Approximation properties of the 
Sz\'asz--Mirakyan operator}}
\label{Section:Szasz-Mirakyan}
%%%%%%%%%%%%%%%%%%%%%%%%%%%%%%%%%%%%%%%%%%%%%%%%%%%%%%%%%%%%%%%%%%%%%%%%%

We start with the approximation properties of the 
Sz\'asz--Mirakyan operator $ \mathcal{P}_n $ 
defined by \eqref{Szasz-Mirakyan}. 
Suppose that  $ \{Y_i=Y_i(x)\}_{i=1}^\infty $ 
is a sequence of independent and identically distributed 
Poisson random variables 
with the parameter $ x \in [0, \infty) $, that is, 
    \[
    \mathbb{P}(Y(x)=k)=e^{-x}\frac{x^k}{k!}, \qquad k=0, 1, 2, \dots.
    \]
Put $ T_n(x):=Y_1(x)+Y_2(x)+\cdots+Y_n(x)$ 
for $ n \in \N $. Then, the distribution of 
$ T_n(x) $ is also Poisson with the parameter $ nx $
by reproductive property.
Hence, we can express \eqref{Szasz-Mirakyan} as 
    \begin{equation}\label{Szasz-expectation}
        \mathcal{P}_n f(x)
        =\E\left[f\left(\frac{1}{n}T_n(x)\right)\right],
        \qquad n \in \N, \, 
        f \in C([0, \infty)), \, x \in [0, \infty).
    \end{equation}
It is convenient to compute several moments of $ T_n(x) $ for later use.  
They are given by
    \begin{equation}\label{Poisson-moments}
    \begin{split}
        \E[T_n(x)]&=nx, \\
        \E[T_n(x)^2]&=(nx)^2+nx, \\
        \E[T_n(x)^3]&=(nx)^3+3(nx)^2+nx, \\
        \E[T_n(x)^4]&=(nx)^4+6(nx)^3+7(nx)^2+nx.
        \end{split}
    \end{equation}
    
We now restrict ourselves to the subspace $ C_\infty([0, \infty)) $. 
Then, we have the following. 

\begin{pr}
If $ f \in C_\infty([0, \infty)) $, then  
$ \mathcal{P}_n f \in C_\infty([0, \infty)) $
and $ \|\mathcal{P}_n f\|_\infty \le \|f\|_\infty $.
\end{pr}

\noindent
Moreover, with the aid of the Poisson distribution, 
Sz\'asz showed in \cite{Szasz} that 
the similar approximation property to 
Proposition \ref{Prop:Bernstein's thm} holds 
for \eqref{Szasz-Mirakyan} as well. 

\begin{pr}[cf.~{\cite[Theorem 3]{Szasz}}]
\label{Prop:Szasz's thm}
For every $ f \in C_\infty([0, \infty)) $, we have
    \[
    \lim_{n \to \infty} 
    \|\mathcal{P}_n f - f\|_\infty=0.
    \]
\end{pr}

We next give an approximation property of 
the Sz\'asz--Mirakyan operator as a linear operator 
acting on the weighted function space. 
Let $ \alpha \ge 1 $ and $ w_\alpha $ the weight function defined by \eqref{weight}.
Note that $ C_\infty([0, \infty)) $
    is dense in $ C_\infty^{w_\alpha}([0, \infty)) $
    for all $ \alpha>1 $. 
    Therefore, every dense subspace of 
    $ C_\infty([0, \infty)) $ is also dense 
    in $ C_\infty^{w_\alpha}([0, \infty)) $. 

Then, we also have the following. 

\begin{pr}[cf.~{\cite[Theorem 6.17]{Altomare}}]
\label{Prop:conv-SM-weighted}
If $ \alpha \ge 1 $, $ n \in \N $
and $ f \in C_\infty^{w_\alpha}([0, \infty)) $, 
then we have 
$ \mathcal{P}_nf \in C_\infty^{w_\alpha}([0, \infty)) $ and 
$ \|\mathcal{P}_nf\|_{\infty, w_\alpha} \le \|f\|_{\infty, w_\alpha} $.
Moreover, we have 
    \[
    \lim_{n \to \infty}
    \|\mathcal{P}_n f - f\|_{\infty, w_\alpha}=0,
    \qquad f \in C_\infty^{w_\alpha}([0, \infty)). 
    \]
\end{pr}

In fact, the latter assertion is an easy consequence of 
the {\it Korovkin-type theorem}. 
Let $ 0<\lambda_1<\lambda_2<\lambda_3 $ and 
$ f_{\lambda_k}(x)=\exp(-\lambda_k x), \, k=1, 2, 3, \, x \in [0, \infty) $. 
Then, one has 
    \[
    \mathcal{P}_n f_{\lambda_k}(x)=\exp\left\{-nx\left(1-\exp
    \left(-\frac{\lambda_k}{n}\right)\right)\right\}, 
    \qquad i=1, 2, 3, \, x \in [0, \infty),
    \]
and $ \|\mathcal{P}_nf_{\lambda_k}-f_{\lambda_k}\|_{\infty, w_\alpha} 
\to 0 $ as $ n \to \infty $
holds for $ k=1, 2, 3 $. 
Since $\{f_{\lambda_1}, f_{\lambda_2}, f_{\lambda_3}\} $
forms a {\it Korovkin set} of $ C_\infty^{w_\alpha}([0, \infty)) $,  
we  apply Proposition \ref{Korovkin-3} to conclude  
$ \|\mathcal{P}_n f - f\|_{\infty, w_\alpha} \to 0 $ as $ n \to \infty $
for every $ f \in C_\infty^{w_\alpha}([0, \infty)) $. 
The reader may consult \cite{Altomare}, 
which is a concise survey on Korovkin-type theorems.

\section{{\bf Proofs of Theorems \ref{Thm:Voronovskaya-Szasz-Mirakyan} and 
\ref{Thm:semigroup-conv-SM}}}
\label{Section:proofs}

In this section, we aim to give the proofs of 
Theorems \ref{Thm:Voronovskaya-Szasz-Mirakyan} and 
\ref{Thm:semigroup-conv-SM} 
by putting an emphasis of the probabilistic representation 
\eqref{Szasz-expectation} of the Sz\'asz--Mirakyan operator. 
At the beginning, we demonstrate the proof of Theorem \ref{Thm:Voronovskaya-Szasz-Mirakyan}.

\begin{proof}[Proof of Theorem {\rm\ref{Thm:Voronovskaya-Szasz-Mirakyan}}]
It follows from \eqref{Poisson-moments} that 
    \[ 
    \E\left[ \left(\frac{1}{n}T_n(x)-x\right) \right]=0, \qquad
    \E\left[ \left(\frac{1}{n}T_n(x)-x\right)^2 \right]
    =\frac{x}{n}.
    \]
By applying the Taylor formula with the integral remainder, one has
    \begin{align}
    \label{P_n-Taylor}
        &\mathcal{P}_nf(x)-f(x) \nn\\
        &=\E\left[ f\left(\frac{1}{n}T_n(x)\right)-f(x)\right] \nn\\
        &=\E\left[f'(x)\left(\frac{1}{n}T_n(x)-x\right)
        +\int_x^{T_n(x)/n}\left(\frac{1}{n}T_n(x)-t\right)f''(t) \, \dd t\right]\nn\\
        &=\E\left[\int_x^{T_n(x)/n}\left(\frac{1}{n}T_n(x)-t\right)
        \big(f''(t) - f''(x)\big) \, \dd t\right]+\frac{x}{2n}f''(x), 
        \qquad x \in [0, \infty),
    \end{align}
where we used 
    \[
    \E\left[\int_x^{T_n(x)/n}\left(\frac{1}{n}T_n(x)-t\right)  \, \dd t\right]
    =\E\left[\frac{1}{2}\left(\frac{1}{n}T_n(x)-x\right)^2\right]=\frac{x}{2n}
    \]
for the fourth line. 
Therefore, \eqref{P_n-Taylor} leads to
    \begin{align}
    &n\left(\mathcal{P}_nf(x)-f(x)\right)-\mathcal{A}f(x) \nn\\
    &=n\E\left[\int_x^{T_n(x)/n}\left(\frac{1}{n}T_n(x)-t\right)
        \big(f''(t) - f''(x)\big) \, \dd t\right]=:n\E[I_n(x)]
        \label{P_n-A}
    \end{align}
for $ n \in \N $ and $ x \in [0, \infty) $.
Since $ f'' \in C_\infty([0, \infty))$
is uniformly continuous,
for $ \ve>0 $, we can find
a sufficiently small $ \delta>0 $ such that 
$ 0<|x-y|<\delta $ implies 
$ |f''(x)-f''(y)|<\ve $.
Hence, one has
    \begin{align}
        &\left|\E\left[I_n(x) \, : \, \left|\frac{1}{n}T_n(x)-x\right|<\delta\right]\right| \nn \\
        &\le \ve\E\left[\left|\int_x^{T_n(x)/n}\left(\frac{1}{n}T_n(x)-t\right) \, \dd t\right| \, 
        : \, \left|\frac{1}{n}T_n(x)-x\right|<\delta\right] \le \frac{\ve x}{2n}.
        \label{P_n-inside}
    \end{align}
On the other hand, it follows from 
the Schwarz inequality that 
    \begin{align}
        &\left|\E\left[I_n(x) \, : \, \left|\frac{1}{n}T_n(x)-x\right| \ge \delta\right]\right| \nn \\
        &\le 2\|f''\|_\infty 
        \E\left[\left|\int_x^{T_n(x)/n}\left(\frac{1}{n}T_n(x)-t\right) \, \dd t\right| \, 
        : \, \left|\frac{1}{n}T_n(x)-x\right| \ge \delta\right] \nn \\
        &\le 2\|f''\|_\infty 
        \E\left[\left|\int_x^{T_n(x)/n}\left(\frac{1}{n}T_n(x)-t\right) \, \dd t\right|^2\right]^{1/2}
        \mathbb{P}\left(
        \left|\frac{1}{n}T_n(x)-x\right| \ge \delta
        \right)^{1/2} \nn \\
        &\le \|f''\|_\infty 
        \E\left[\left(\frac{1}{n}T_n(x)-x\right)^4\right]^{1/2}
        \mathbb{P}\left(
        \left|\frac{1}{n}T_n(x)-x\right| \ge \delta
        \right)^{1/2}. 
        \label{P_n-outside}
    \end{align}
By using \eqref{Poisson-moments}, we see that 
    \begin{equation}\label{P_n-B}
    \E\left[\left(\frac{1}{n}T_n(x)-x\right)^4\right]^{1/2}
    =\left(\frac{3x^2}{n^2}+\frac{x}{n^3}\right)^{1/2}
    \le \frac{\sqrt{3}x}{n}
    +\frac{\sqrt{x}}{n\sqrt{n}}. 
    \end{equation}
Moreover, by applying 
\cite[Theorem 1]{C}, we have 
    \begin{equation}\label{P_n-C}
    \mathbb{P}\left(
        \left|\frac{1}{n}T_n(x)-x\right| \ge \delta
        \right)
    \le 2\exp\left(-\frac{n\delta^2}{2(x+\delta)}\right).
    \end{equation}
Therefore, by combining \eqref{P_n-A} with 
\eqref{P_n-inside}, \eqref{P_n-outside}, \eqref{P_n-B}, and \eqref{P_n-C},
we obtain 
    \[
    \begin{aligned}
    &w_{\alpha}(x)\left|n\left(\mathcal{P}_nf(x)-f(x)\right)
    -\mathcal{A}f(x)\right| \\
    &\le \frac{\ve }{2}xw_\alpha(x)+
    \sqrt{2}\|f''\|_\infty
    \left(\sqrt{3}xw_\alpha(x)
    +\frac{1}{\sqrt{n}}\sqrt{x}w_\alpha(x)\right)
    \exp\left(-\frac{n\delta^2}{2(x+\delta)}\right)
    \end{aligned}
    \]
for $ n \in \N $ and $ x \in [0, \infty). $
Here, we note that
    \[
    \sup_{x \in [0, \infty)}xw_\alpha(x)
    =\frac{\alpha-1}{\alpha}\left(\frac{1}{\alpha-1}\right)^{1/\alpha}, \quad 
    \sup_{x \in [0, \infty)}\sqrt{x}w_\alpha(x)
    =\frac{2\alpha-1}{2\alpha}\left(\frac{1}{2\alpha-1}\right)^{1/2\alpha}<\infty.
    \]
Furthermore, it follows from a direct calculus that the function 
$ xw_\alpha(x)\exp({-\frac{n\delta^2}{4(x+\delta)}}) $
has a unique maximizer $x_* \in (0, \infty)$. 
Therefore, we have 
    \[
    \begin{aligned}
    &\|n(\mathcal{P}_nf-f)
    -\mathcal{A}f\big\|_{\infty, w_\alpha} \\
    &\le \frac{\ve(\alpha-1)}{2\alpha}\left(\frac{1}{\alpha-1}\right)^{1/\alpha}\\
    &\hspace{1cm}+
    \sqrt{2}\|f''\|_\infty
    \left(\sqrt{3}x_*w_\alpha(x_*)
    \exp\left(-\frac{n\delta^2}{2(x_*+\delta)}\right)
    +\frac{2\alpha-1}{2\alpha\sqrt{n}}\left(\frac{1}{2\alpha-1}\right)^{1/2\alpha}\right)\\
    &\to \frac{\ve(\alpha-1)}{2\alpha}\left(\frac{1}{\alpha-1}\right)^{1/\alpha}
    \end{aligned}
    \]
as $n \to \infty$. Since $ \ve>0 $ is arbitrary, 
we obtain the desired convergence by letting $\ve \searrow 0$. 

Next, suppose that $ f'' $ is Lipschitz and $ \alpha>3/2 $. 
Then, it follows from \eqref{P_n-B} that  
    \begin{align}
    &\big|n\big(\mathcal{P}_nf(x)-f(x)\big)-\mathcal{A}f(x)\big| \nn\\
    &\le n\mathrm{Lip}(f'')
    \E\left[\int_{x \wedge T_n(x)/n}^{x \vee T_n(x)/n}\left|\frac{1}{n}T_n(x)-t\right|
        |t-x| \, \dd t\right] \nn \\
    &\le \frac{1}{6}n\mathrm{Lip}(f'')
    \E\left[ \Big|\frac{1}{n}T_n(x)-x\Big|^3\right] \nn \\
    &\le \frac{1}{6}n\mathrm{Lip}(f'')
    \E\left[ \Big|\frac{1}{n}T_n(x)-x\Big|^4\right]^{3/4} \nn\\
    &\le \frac{1}{6}n\mathrm{Lip}(f'')
    \left(\frac{3x^2}{n^2}+\frac{x}{n^3}\right)^{3/4}
    \le \frac{1}{6}\mathrm{Lip}(f'')
    \left(\frac{3^{3/4}x^{3/2}}{\sqrt{n}}+\frac{x^{3/4}}{\sqrt{n}}\right).\nn
    \end{align}
By noting 
    \[
    \begin{aligned}
    &\sup_{x \in [0, \infty)}(3^{3/4}x^{3/2}+x^{3/4})
    w_\alpha(x)\\
    &=\frac{3^{3/4}(2\alpha-3)}{2\alpha}\left(\frac{3}{2\alpha-3}\right)^{3/2\alpha}
    +\frac{4\alpha-3}{4\alpha}\left(\frac{3}{4\alpha-3}\right)^{3/4\alpha}<\infty, 
    \end{aligned}
    \]
we obtain 
    \[
    \|n(\mathcal{P}_nf-f)
    -\mathcal{A}f\big\|_{\infty, w_\alpha} 
    \le \frac{1}{6\sqrt{n}}M_\alpha\mathrm{Lip}(f''), \qquad n \in \N,
    \]
which is the desired estimate.  
\end{proof}

As a combination of Theorem \ref{Thm:Voronovskaya-Szasz-Mirakyan}, 
Trotter's approximation theorem and a technique demonstrated in \cite{AC}, 
we can give the proof of Theorem \ref{Thm:semigroup-conv-SM}. 
For more details on Trotter's approximation theorem, 
see Appendix A.

\begin{proof}[Proof of Theorem \ref{Thm:semigroup-conv-SM}]
Let $ \alpha>1 $.  Recall that 
$ C_\infty^2([0, \infty)) \subset D(\mathcal{A}) $ is dense in 
$ C_\infty([0, \infty)) $
by applying the (generalization of) Stone--Weierstrass theorem. 
Hence, it is also dense 
in $ C_\infty^{w_\alpha}([0, \infty)) $. 
Moreover, by virtue of the proof of \cite[Lemma 4.2]{AC2}, 
we can show that 
$ (I-\mathcal{A})(C_\infty^2([0, \infty))) $ is dense 
in $ C_\infty([0, \infty)) $ and 
so is in $ C_\infty^{w_\alpha}([0, \infty)) $. 

We now consider the linear operator 
$\mathcal{A}_*$ defined by 
    \[
    \mathcal{A}_*f
    :=\lim_{n \to \infty} n(\mathcal{P}_n f - f), 
    \qquad f \in D(\mathcal{A}_*),
    \]
where the domain $ D(\mathcal{A}_*) $
of $ \mathcal{A}_* $ is given by 
    \[
    D(\mathcal{A}_*)
    :=\Big\{f \in C_\infty^{w_\alpha}([0, \infty)) \, 
    \Big| \, \lim_{n \to \infty} n(\mathcal{P}_n f - f)
    \text{ exists in }C_\infty^{w_\alpha}([0, \infty))\Big\}.
    \]
Thanks to $ C_\infty^2([0, \infty)) \subset D(\mathcal{A}_*) $, we know that 
$ D(\mathcal{A}_*) $ is dense 
in $ C_\infty^{w_\alpha}([0, \infty)) $. 
Moreover, $ (I-\mathcal{A})(C_\infty^2([0, \infty))) $
coincides with $ (I-\mathcal{A}_*)(C_\infty^2([0, \infty))) $ and it is dense in $ C_\infty^{w_\alpha}([0, \infty)) $. 
Therefore, Trotter's approximation theorem 
(see Theorem \ref{Prop:Trotter}) implies that 
the closure $ (\overline{\mathcal{A}_*}, D(\overline{\mathcal{A}_*})) $ of $ (\mathcal{A}_*, D(\mathcal{A}_*)) $ generates a $C_0$-semigroup
$ (\mathsf{S}_t)_{t \ge 0} $ on $ C_\infty^{w_\alpha}([0, \infty)) $ and we obtain 
\[
    \lim_{n \to \infty}
    \|\mathcal{P}_n^{[nt]} f 
    - \mathsf{S}_t f\|_{\infty, w_\alpha}=0
    \]
for every $ f \in C_\infty^{w_\alpha}([0, \infty)) $ 
and $ t \ge 0 $. 
On the other hand, if $ f \in C_\infty^2([0, \infty)) $, 
we then see that the $ \overline{\mathcal{A}_*}f=\mathcal{A}f $
and therefore it turns out that 
    \[
    (I-\overline{\mathcal{A}_*})(C_\infty^2([0, \infty))=
    (I-\mathcal{A})(C_\infty^2([0, \infty))
    \]
is dense in $ C_\infty^2([0, \infty)) $. 
Particularly, the linear operators 
$ I-\overline{\mathcal{A}_*} $ and $ I-\mathcal{A} $
are invertible and $ C_\infty^2([0, \infty)) $
is a core for both $ \overline{\mathcal{A}_*} $ and $ \mathcal{A} $. Thus, we obtain 
$ D(\overline{\mathcal{A}_*})=D(\mathcal{A}) $
and $ \overline{\mathcal{A}_*}=\mathcal{A} $. 
This concludes $ {\sf{S}}_t={\sf{T}}_t $ for $ t \ge 0 $. 

The latter part is readily obtained from 
Proposition \ref{Prop:Trotter-rate}
once we put two semi-norms $ \varphi_n $ and $ \psi_n $ as 
    \[
    \varphi_n(f):=\|\mathcal{A}f\|_{\infty, w_\alpha}
    +\frac{1}{6\sqrt{n}}M_\alpha\mathrm{Lip}(f''), \quad 
    \psi_n(f):=\frac{1}{6\sqrt{n}}M_\alpha\mathrm{Lip}(f''), 
    \]
for $ f \in C_\infty^2([0, \infty)).  $
\end{proof}

In \cite{AC}, Altomare and Carbone dealt with 
a degenerate differential operator 
$ \mathcal{A}f(x)=a(x)f''(x) $ 
on the weighted function space 
$ C_\infty^{w_\alpha}([0, \infty)), \, \alpha \ge 1, $ and show that 
it generates a $C_0$-semigroup of positive linear operators, where the continuous function $ a(x) $
satisfies 
    \[
    \lim_{x \to 0+}a(x)f''(x)
    =\lim_{x \to \infty}w_\alpha(x)(a(x)f''(x))=0.
    \]
Moreover, they apply the result to deduce that 
the iterates of the (generalized) Sz\'asz--Mirakyan operator uniformly converges 
to the $C_0$-semigroup in a purely functional-analytic approach. 
However, our limit theorems together with 
its rate of convergence are based on not only 
a functional-analytic view but also a probabilistic one. 
Therefore, we come to capture an interesting 
characterization of the limiting semigroup 
in terms of a certain diffusion process, 
as in the next section.

\section{{\bf A probabilistic interpretation 
of Theorem \ref{Thm:semigroup-conv-SM}}}
\label{Section:probab-approach}

As is seen in \eqref{Bernstein-Markov chain},
the iterate of the Bernstein operator 
is represented as the expectation of a certain 
time-homogeneous Markov chain. 
Thanks to the representation, we obtain 
a characterization of the convergence of 
the iterates of the Bernstein operator 
in terms of the diffusion process given by 
the weak limit of the Markov chain 
(see Proposition \ref{Prop:KYZ-weak-conv}). 

We also give such a characterization of 
Theorem \ref{Thm:semigroup-conv-SM}
by using a certain discrete Markov chain constructed by 
the iterates of the Sz\'asz--Mirakyan operator. 
Let $ G_n : [0, \infty] \to [0, \infty]  $ 
be a random function defined by 
    $$
    G_n(x):=\frac{1}{n}T_n(x), \qquad x \in [0, \infty),
    $$
and $ \{G_n^k\}_{k=1}^\infty $ a sequence of 
independent copies of $ G_n $.
For $ k \in \N, $ we also define a random function 
$ H_n^k : [0, \infty) \to [0, \infty) $ by 
    $$
    H_n^k(x):=(G_n^k \circ G_n^{k-1} \circ \cdots \circ G_n^1)(x)
    , \qquad x \in [0, \infty).
    $$
Then, by the definition 
of $\mathcal{P}_n$, we see that 
    \begin{equation}
    \label{SM-Markov chain}
    \mathbb{E}[f(H_n^k(x))]
    =\mathcal{P}_n^kf(x), \qquad x \in [0, \infty). 
    \end{equation}
For every $ x \in [0, \infty), $
we observe that the sequence 
$ \{H_n^k(x)\}_{k=1}^\infty $ is a time-homogeneous Markov chain 
with values in $ \{i/n \, | \, i =0, 1, 2, \dots\} $
with the one-step transition probability 
    $$
    p\left(\frac{i}{n}, \frac{j}{n}\right)=
    \mathbb{P}\left(
    H_n^{k+1}(x)=\frac{j}{n} \, \Bigg| \, 
    H_n^k(x)=\frac{i}{n}\right)
    =e^{-i}\frac{i^j}{j!}, 
    \qquad i, j=0, 1, 2, \dots.
    $$
We note that, since $p(0, 0)=1$ and $ p(0, i/n)=0 $ for $i=1, 2, \dots$, 
the state 0 is an absorbing state of $ \{H_n^k(x)\}_{k=1}^\infty $. 
By using the representation \eqref{SM-Markov chain}, 
Theorem \ref{Thm:semigroup-conv-SM} reads 
    $$
    \lim_{n \to \infty}\sup_{x \in [0, \infty)}
    w_\alpha(x)
    \big|\E\big[f(H_n^{[nt]}(x))\big] - \mathsf{P}_tf(x)\big|=0
    $$
for $ \alpha>1, \, f \in C_\infty^{w_\alpha}([0, \infty))$ and $ t \ge 0$ . 
We now give another representation of 
the limiting semigroup $ (\mathsf{P}_t)_{t \ge 0} $ in terms of a certain diffusion process. 

\begin{tm}\label{Thm:SM-diffusion}
The $C_0$-semigroup $ (\mathsf{P}_t)_{t \ge 0} $
acting on $ C_\infty^{w_\alpha}([0, \infty)) $
is represented as 
    \[
    \mathsf{P}_tf(x)
    =\E\big[f(\mathbf{Y}_t(x))\big], \qquad 
    x \in [0, \infty), \, t \ge 0, 
    \]
where $ (\mathbf{Y}_t(x))_{t \ge 0} $
is a diffusion process which is the unique strong solution 
to the stochastic differential equation 
    \begin{equation}\label{SM-SDE}
    \dd\mathbf{Y}_t(x)=\sqrt{\mathbf{Y}_t(x)} \, \dd W_t, 
    \qquad \mathbf{Y}_0(x)=x \in [0, \infty),
    \end{equation}
with $ (W_t)_{t \ge 0} $ being a one-dimensional standard Brownian motion.
\end{tm}

\begin{proof}
We first verify that the linear operator $\mathcal{A}$ 
given by \eqref{generator-SM} satisfies 
the  {\it positive maximal principle}, 
that is, for $ f \in C_\infty^2([0, \infty)) $, 
the condition $ f(x_0)=\sup_{x \in [0, \infty)}f(x) \ge 0 $ 
for some $ x_0 \in [0, \infty) $ 
implies $ \mathcal{A}f(x_0)= (x_0/2)f''(x_0) \le 0 $.
Hence, by applying the Hille--Yosida theorem 
(cf.~\cite[Theorem 19.11]{Kallen}), 
we know that the closure of 
$ \big(\mathcal{A}, C_\infty^2([0, \infty))\big) $
generates a unique Feller semigroup $ (\mathsf{P}_t)_{t \ge 0} $
on $ C_\infty([0, \infty)) $. 

On the other hand, we easily have
    \begin{align}
    n\mathbb{E}[(H_n^{k+1}(x)-H_n^{k}(x))^2 \, 
    | \, H_n^{k}(x)=y]
    &=n\mathbb{E}[(G_n(y)-y)^2]=y,  
    \label{MC-1}\\
    n\mathbb{E}[(H_n^{k+1}(x)-H_n^{k}(x)) \, 
    | \, H_n^{k}(x)=y]
    &=n\mathbb{E}[(G_n(y)-y)]=0. 
    \label{MC-2}
    \end{align}
Furthermore, for $\ve>0$, $ R>0 $ and $ 0 \le y<R $, one has
    \begin{align}\label{MC-3}
    &\mathbb{P}(|H_n^{k+1}(x)-H_n^{k}(x)|>\ve \, | \, H_n^k(x)=y)\nn\\
    &=\mathbb{P}(|G_n(y)-y|>\ve) 
    \le 2\exp\left(-\frac{n\delta^2}{2(y+\delta)}\right)
    < 2\exp\left(-\frac{n\delta^2}{2(R+\delta)}\right)
    \end{align}
by using \eqref{P_n-C}. 
Hence, it follows from \eqref{MC-1}, \eqref{MC-2} and 
\eqref{MC-3} that 
$ (H_n^{[nt]}(x))_{t \ge 0}, \, n=1, 2, \dots, $
converges weakly to the diffusion process $ ({\bf Y}_t(x))_{t \ge 0} $
which solves \eqref{SM-SDE} as $ n\to \infty $. 
Here, we have applied the convergence criteria given in 
\cite[Theorem 11.2.3]{SV}. 
In particular, we conclude that 
    \[
    \lim_{n \to \infty}\sup_{x \in [0, \infty)}
    \big|\E\big[f(H_n^{[nt]}(x))\big] - \E\big[f({\bf Y}_t(x))\big]\big|=0
    \]
for $ f \in C_\infty([0, \infty)), \, t \ge 0, $ and 
    \begin{equation}\label{MC-4}
    {\sf P}_tf(x)=\E\big[f({\bf Y}_t(x))\big], 
    \qquad f \in C_\infty([0, \infty)), \, t \ge 0, \, x \in [0, \infty). 
    \end{equation}
Since $ C_\infty([0, \infty)) $ is dense in 
$ C_\infty^{w_\alpha}([0, \infty)) $ and 
each $ {\sf P}_t, \, t \ge 0, $ is bounded, 
the representation \eqref{MC-4} 
can be extended to $ C_\infty^{w_\alpha}([0, \infty)) $ as well. 
\end{proof}

%%%%%%%%%%%%%%%%%%%%%%%%%%%%%%%%%%%%%%%%%%%%%%%%%%%%%%%%%%%%%%%%%%%%%%%%%
\section{{\bf Conclusions and further directions}}
\label{Section:Conclusion}
%%%%%%%%%%%%%%%%%%%%%%%%%%%%%%%%%%%%%%%%%%%%%%%%%%%%%%%%%%%%%%%%%%%%%%%%%

Throughout the present paper, we have discussed 
limit theorems for the iterates of the
Sz\'asz--Mirakyan operator  $ \mathcal{P}_n $ 
and have shown that the iterates of
$ \mathcal{P}_n $ uniformly converges to 
the diffusion semigroup generated
by a second order degenerate differential operator. 
In fact, there are a few papers discussing limit theorems 
for the iterates of $ \mathcal{P}_n $ being
essentially same as our main theorems (see e.g., \cite{AC}).
Nonetheless, we cannot find any references in which 
the explicit characterization of the limiting semigroup $ ({\sf P}_t)_{t \ge 0} $
in terms of a diffusion process is obtained. 
The present paper makes a significant contribution 
to probability theory as well as approximation theory 
in that we find a stochastic phenomenon
behind a positive linear operator 
approximating certain continuous functions on the half line.  

As pointed out in \cite{KZ}, 
a positive linear operator represented as the expectation of 
a certain random variable like \eqref{Bernstein op} 
and \eqref{Szasz-Mirakyan} may uniformly approximate 
certain continuous functions. 
Indeed, yet another example of such a linear operator was introduced
in \cite{Baskakov}. For $n \in \N$,  a linear operator $V_n$
acting on $ C([0, \infty)) $ defined by 
    \[
    V_n f(x):=\sum_{k=0}^\infty \binom{n+k-1}{k}\frac{x^k}{(x+1)^{n+k}}
    f\left(\frac{k}{n}\right), \qquad f \in C([0, \infty)), \, x \in [0, \infty),
    \]
is called the {\it Baskakov operator}. As is easily seen, this operator is 
defined in terms of the negative binomial distribution. 
We also refer to \cite{Becker} for the approximation property of $ V_n $. 
By heuristic calculation, we observe that the linear operator given by
$ \mathcal{A}f(x)=(x(x+1)/2)f''(x) $ appears as the limit of 
$ \{n(V_n-I)\}_{n=1}^\infty $. 
We then expect a new example of limit theorems 
which captures another kind of diffusion process as the limit. 
Moreover, by highly generalizing these linear operators, 
we can define positive linear operators approximating 
certain class of continuous functions on a locally compact $ E \subset \R $ 
associated with probability measures on $E$. 
The approximating properties and the convergences of the iterates 
of the generalized operators to some diffusion semigroups 
from probabilistic perspectives are to be 
discussed in the forthcoming paper. 

On the other hand, some papers concern with multidimensional generalizations 
of the classical Bernstein operator and discuss their limit theorems, 
to say nothing of the approximating properties. 
See e.g., \cite[Theorem 2.3]{CT} for a limit theorem for iterates of the
multidimensional Bernstein operator on a simplex.
Despite such developments, as far as we know, 
multidimensional generalizations have not been discussed sufficiently 
except for the case of the Bernstein operator. 
Therefore, it should be an interesting problem to define 
multidimensional generalizations of existing positive linear operators 
such as Sz\'asz--Mirakyan and Baskakov operators and 
to investigate their fundamental properties. 

Furthermore, it is also intriguing to consider the infinite-dimensional generalizations. 
In fact, there exists an infinite-dimensional analogue of 
the multidimensional Wright--Fisher diffusion, 
which is known as the (measure-valued) {\it Fleming--Viot process}. 
See e.g., \cite{EK1} for more details. 
Since the Wright--Fisher diffusion is viewed as the limit of the iterates of 
the Bernstein operator, we expect that there is an ``operator'' on the space of 
probability measures whose ``iterates'' converges to the diffusion semigroup
corresponding to the Fleming--Viot process in some sense. 
If we find such  ``operators'' in several infinite-dimensional settings, 
we can also construct several measure-valued diffusion processes 
through some limit theorems for their iterates.

%%%%%%%%%%%%%%%%%%%%%%%%%%%%%%%%%%%%%%%%%%%%%%%%%%%%%%%%%%%%%%%%%%%%%%%%%
%%%%%%%%%%%%%%%%%%%%%%%%%%%%%%%%%%%%%%%%%%%%%%%%%%%%%%%%%%%%%%%%%%%%%%%%%
\begin{appendix}
\label{app}

%%%%%%%%%%%%%%%%%%%%%%%%%%%%%%%%%%%%%%%%%%%%%%%%%%%%%%%%%%%%%%%%%%%%%%%%%
\section{{\bf Trotter's approximation theorem and 
its rate of convergence}}
%%%%%%%%%%%%%%%%%%%%%%%%%%%%%%%%%%%%%%%%%%%%%%%%%%%%%%%%%%%%%%%%%%%%%%%%%

Trotter's approximation theorem provides a sufficient condition that 
the iterates of a bounded linear operator 
acting on a Banach space converges to a $ C_0 $-semigroup. 
We give the statement of Trotter's approximation theorem 
when the iteration of a linear operator enjoys the contraction property.

\begin{pr}[cf.~{\cite[Theorem 5.1]{Trotter}}, {\cite[Theorem 2.13]{Kurtz}}]
\label{Prop:Trotter}
Let $ (\mathcal{U}, \|\cdot\|_{\mathcal{U}}) $ be a 
Banach space. Suppose that a sequence $ \{T_n\}_{n=1}^\infty $ 
of bounded linear operators on $ \mathcal{U} $ satisfies that 
$ \|T_n\| \le 1 $ for $ n \in \N $. 
Put $ \mathcal{A}_n:=n(T_n-I), \, n \in \N. $
We define a linear operator $ \mathcal{A} $ 
by the closure of the limit of $ \mathcal{A}_n $.  
If the domain $ D(\mathcal{A}) $ is dense in $ \mathcal{U} $ and 
the range of $ \lambda - \mathcal{A} $ is dense 
in $ \mathcal{U} $ for some $ \lambda>0 $, 
then  there exists a 
$ C_0 $-semigroup $ (\mathsf{T}_t)_{t \ge 0} $
acting on $ \mathcal{U} $ 
satisfying $\|\mathsf{T}_t\| \le 1, \, t \ge 0$, and 
    \begin{equation}\label{limit-Trotter}
    \lim_{n \to \infty}
    \|T_n^{[nt]}f - \mathsf{T}_tf\|_{\mathcal{U}}=0,
    \qquad t \ge 0, \, f \in \mathcal{U}. 
    \end{equation}
\end{pr}

On the other hand, as is easily seen,
Proposition \ref{Prop:Trotter} does not imply 
any quantitative estimates of \eqref{limit-Trotter}. 
Campiti and Tacelli established in \cite{CT} 
a refinement of Proposition \ref{Prop:Trotter} by giving the rate of convergence 
of \eqref{limit-Trotter}.

\begin{pr}[cf.~{\cite[Theorem 1.1]{CT}, see also \cite{CTa}}]
\label{Prop:Trotter-rate}
Suppose that  $ \mathcal{U} $, $ \{T_n\}_{n=1}^\infty $, 
$ \{\mathcal{A}_n\}_{n=1}^\infty $ and $ (\mathsf{T}_t)_{t \ge 0} $
are as in Proposition {\rm \ref{Prop:Trotter}}.
Let $ \mathcal{D} $ be a dense subspace of $ \mathcal{U} $. 
We assume that
    \[
    \|\mathcal{A}_n f\|_{\mathcal{U}} \le \varphi_n(f), \qquad 
    \|\mathcal{A}_n f - \mathcal{A}f\|_{\mathcal{U}} \le \psi_n(f), 
    \qquad f \in \mathcal{D},
    \]
where $ \varphi_n, \psi_n : \mathcal{D} \to [0, \infty) $
are semi-norms with $ \psi_n(f) \to 0 $ as $ n \to \infty $
for $ f \in \mathcal{D} $.
Then, for  $ t \ge 0 $ and 
$ f \in \{g \in \mathcal{D} \, 
| \, \mathsf{T}_t g \in \mathcal{D}, \, t \ge 0\} $, we have 
    \begin{equation}\label{limit-Trotter2}
    \|T_n^{[nt]}f - \mathsf{T}_tf\|_{\mathcal{U}}
    =\sqrt{\frac{t}{n}}\varphi_n(f)+\frac{1}{n}\varphi_n(f)
    +\int_0^t \psi_n(\mathsf{T}_sf) \, \dd s,
    \qquad n \in \N. 
    \end{equation}
\end{pr}

Indeed, we have used Proposition \ref{Prop:Trotter-rate}
in order to deduce the rate of convergence of the 
iterates of the Sz\'asz--Mirakyan operator (see Theorem \ref{Thm:semigroup-conv-SM}).

\section{{\bf Korovkin-type theorems}}
\label{Section:Korovkin}

It is well-known that the Korovkin-type theorems
provide a quite powerful sufficient condition to 
deduce that a sequence of positive linear operators acting 
on some function spaces converges strongly to the identity operator. 
Korovkin's first theorem, which was first discovered by Korovkin himself in \cite{Korovkin}, 
is stated as follows:

\begin{pr}[Korovkin's first theorem]
\label{Korovkin}
Let $ e_0(x) \equiv 1 $, $ e_1(x)=x $ and $ e_2(x)=x^2 $
for $ x \in [0, 1] $. 
If a sequence $ \{L_n\}_{n=1}^\infty $ of positive linear operators 
acting on $ C([0, 1]) $ satisfies
    \[
    \lim_{n \to \infty}
    \|L_ne_i - e_i\|_\infty=0, \qquad i=0, 1, 2,
    \]
then, it holds that 
    \[
    \lim_{n \to \infty}
    \|L_nf - f\|_\infty=0, \qquad f \in C([0, 1]). 
    \]
\end{pr}

So far, a number of generalizations of 
Proposition \ref{Korovkin} have been established 
in various settings. We refer to
\cite{Altomare} and \cite{AC3}
for good surveys of this topic. 
In the sequel, $ (\mathfrak{X}, \|\cdot\|_{\mathfrak{X}}) $ is used 
in order to represent  Banach spaces of some continuous functions on
a locally compact space $E \subset \R$. 
We give a general formulation of the Korovkin-type theorem.

\begin{df}[Korovkin set]
A subset $\mathcal{K}$ of $\mathfrak{X}$ is called a Korovkin set of $\mathfrak{X}$
if for every sequence $ \{L_n\}_{n=1}^\infty $ of positive linear operators
acting on $\mathfrak{X}$ satisfying $\sup_{n \in \N}\|L_n\|<\infty$ and 
    \[
    \lim_{n \to \infty}\|L_nf - f\|_{\mathfrak{X}}=0, \qquad f \in \mathcal{K},
    \]
    then, it holds that 
    \[
    \lim_{n \to \infty}\|L_nf - f\|_{\mathfrak{X}}=0, \qquad f \in \mathfrak{X}.
    \]
\end{df}

Namely, Korovkin-type results aim to find out which functions form Korovkin sets of 
a given function space.
Note that, in other words, Korovkin's first theorem asserts that 
the set $ \{e_0, e_1, e_2\} \subset C([0, 1]) $ is a Korovkin set of $ C([0, 1]) $.
We next give several examples of Korovkin sets of $ C_\infty([0, \infty)) $.

\begin{pr}[cf.~{\cite[Corollary 6.7]{Altomare}}]
\label{Korovkin-1}
Suppose that  $ 0<\lambda_1<\lambda_2<\lambda_3 $. Then,
$\{f_{\lambda_1}, f_{\lambda_2}, f_{\lambda_3}\}$ is a 
Korovkin set of $  C_\infty([0, \infty)) $, 
where $f_{\lambda_k}(x):=\exp(-\lambda_k x)$ for $k=1, 2, 3$.
\end{pr}

Moreover, we obtain an easy way to find out 
Korovkin sets of $ C_\infty^w([0, \infty)) $ when 
the weight function $w$ is supposed to be bounded.

\begin{pr}
[cf.~{\cite[Proposition 6.16]{Altomare}}]
\label{Korovkin-3}
Let $ w \in C_b([0, \infty)) $. 
Then, every Korovkin set of $ C_\infty([0, \infty)) $
is also a Korovkin set of $C_\infty^w([0, \infty))$. 
In particular, the set $\{f_{\lambda_1}, f_{\lambda_2}, f_{\lambda_3}\}$ 
as in Proposition {\rm \ref{Korovkin-1}} is a Korovkin set of 
$C_\infty^w(E[0, \infty)$. 
\end{pr}

\end{appendix}
%%%%%%%%%%%%%%%%%%%%%%%%%%%%%%%%%%%%%%%%%%%%%%%%%%%%%%%%%%%%%%%%%%%%%%%%%
%%%%%%%%%%%%%%%%%%%%%%%%%%%%%%%%%%%%%%%%%%%%%%%%%%%%%%%%%%%%%%%%%%%%%%%%%

%%%%%%%%%%%%%%%%%%%%%%%%%%%%%%%%%
%%%%%%%%%%%%%%%%%%%%%%%%%%%%%%%%%
%%%%%%%%%%%%%%%%%%%%%%%%%%%%%%%%%

%%%%%%%%%%%%%%%%%%%%%%%%%%%%%%%%%
%%%%%%%%%%%%%%%%%%%%%%%%%%%%%%%%%
%%%%%%%%%%%%%%%%%%%%%%%%%%%%%%%%%

\end{document}